\documentclass[11pt]{amsart}
\RequirePackage{amsmath}
\RequirePackage{amsthm}
\RequirePackage{amsfonts}
\RequirePackage{amssymb}
\RequirePackage{multirow}
\RequirePackage{multicol}
\RequirePackage{graphicx}
\RequirePackage{float}
\RequirePackage{enumerate} 
\RequirePackage{cancel}
\RequirePackage{appendix}
\RequirePackage{calrsfs}
\RequirePackage{cite}
\RequirePackage[figurewithin=section]{caption}
\RequirePackage{xcolor}
\RequirePackage{hyperref}
\RequirePackage{mathtools}

\RequirePackage{tikz}
\usetikzlibrary{babel}
\usetikzlibrary{shapes,arrows}
\tikzstyle{block} = [draw, fill=blue!20, rectangle, 
    minimum height=3em, minimum width=3em]
\tikzstyle{sum} = [draw, fill=blue!20, circle, node distance=1cm]
\tikzstyle{input} = [coordinate]
\tikzstyle{output} = [coordinate]

\numberwithin{equation}{section}

\newtheorem{theorem}{Theorem}[section]
\newtheorem{proposition}[theorem]{Proposition}
\newtheorem{corollary}[theorem]{Corollary}
\newtheorem{lemma}[theorem]{Lemma}

\newtheorem{thm}{Theorem}

\theoremstyle{definition}

\newtheorem{definition}[theorem]{Definition}

\newtheorem{remark}[theorem]{Remark}

\newcommand{\C}{\mathbb{C}}

\newcommand{\R}{\mathbb{R}}
\newcommand{\D}{\mathbb{D}}
\newcommand{\N}{\mathbb{N}}

\renewcommand{\H}{\mathbb{H}}

\newcommand{\abs}[1]{\left| #1 \right|}

\title[Extremal rate]{Extremal rate of convergence in discrete\\
hyperbolic and parabolic dynamics}

\author[F. J. Cruz-Zamorano]{Francisco J. Cruz-Zamorano}
\address{Departamento de An\'alisis Matem\'atico, Facultad de Ciencias, Universidad de La Laguna, Avenida Astrof\'isico Francisco Sanchez S/N, 38206 San Crist\'obal de La Laguna, Santa Cruz de Tenerife, Spain}
\email{fcruzzam@ull.edu.es}

\author[K. Zarvalis]{Konstantinos Zarvalis}
\address{Department of Mathematics, Aristotle University of Thessaloniki, 54124, Thessaloniki, Greece}
\email{zarkonath@math.auth.gr}

\date{}
\thanks{Cruz-Zamorano was supported by Ministerio de Innovaci\'on y Ciencia, Spain, project PID2022-136320NB-I00 and by Ministerio de Universidades, Spain, through the action Ayuda del Programa de Formaci\'on de Profesorado Universitario, reference FPU21/00258. \\
Zarvalis was partially supported by Junta de Andaluc\'{i}a, grant number QUAL21 005 USE}

\begin{document}
\begin{abstract}
This paper investigates the dynamical behaviour of holomorphic self-maps of the upper half-plane. More precisely, we focus on the hyperbolic and parabolic self-maps whose orbits approach the Denjoy--Wolff point with the slowest possible rate. We characterize self-maps of such extremal rate using various tools, like the Herglotz representation, the conformality of the Koenigs function at the Denjoy--Wolff point and the hyperbolic distance.
\end{abstract}

\maketitle

\section{Introduction}
The study of the iterative behaviour of holomorphic self-maps $f$ of a domain $\Omega \subset \C$ is a well-established topic of research within Complex Dynamics. When $\Omega$ is the upper half-plane $\H \vcentcolon = \{z \in \C \vcentcolon \mathrm{Im}(z) > 0\}$, the celebrated Denjoy--Wolff Theorem \cite[Theorem 1.8.4]{BCDM} explains the global dynamical behaviour of $f$. More precisely, this theorem asserts that for every non-elliptic holomorphic self-map $f \vcentcolon \H \to \H$ (i.e. $f(z) \neq z$ for all $z \in \H$), the sequence of iterates defined by $f^0 = \mathrm{Id}_{\H}$ and $f^n = f^{n-1} \circ f$, $n \in \N$, converges locally uniformly in $\H$ to a point $\tau \in \R \cup \{\infty\}$ (known as the \textit{Denjoy--Wolff point} of $f$), as $n \to +\infty$. Throughout this article, we will consistently assume that $\tau = \infty$. This normalization can always be achieved by conjugating $f$ with a suitable automorphism of $\H$.

Several recent articles inspect how fast the discrete orbits $\{f^n(z) : n \in \N\cup\{0\}\}$, $z\in\H$, escape to the Denjoy--Wolff point infinity \cite{DB-Rate-H,DB-Rate-P,BCDM-Rate,BCZZ,BK-Speed,Bracci,BCK,CZ-FiniteShift,KTZ}. Namely, they obtain estimates for $\abs{f^n(z)}$ in terms of various dynamical properties of $f$. As we will review in a moment, for different classes of non-elliptic maps, there are known sharp universal lower bounds for this rate. This paper is concerned with what we call the \textit{extremal rate of convergence}, which occurs when the actual rate matches these lower bounds (that is, the convergence of the orbit towards the Denjoy--Wolff point is the slowest possible). Our primary goal is to provide comprehensive characterizations for when a non-elliptic self-map exhibits such an extremal rate. These characterizations will be developed using a variety of tools from Complex Analysis, including the Herglotz representation, the notion of conformality at a boundary point, the asymptotic behaviour of the maps themselves near infinity, hyperbolic geometry, and composition operators.

To elaborate, non-elliptic self-maps with Denjoy--Wolff point infinity are classified based on the angular derivative $f'(\infty) \vcentcolon = \angle\lim_{z\to\infty} (f(z)/z) \in [1,+\infty)$ \cite[Corollary 2.5.5]{Abate}. If $\alpha \vcentcolon = f'(\infty) > 1$, the map $f$ is said to be \textit{hyperbolic}. For these self-maps, Valiron \cite{Valiron} showed that orbits converge non-tangentially, and that the dynamics can be linearized by a Koenigs function $h \vcentcolon \H \to \C$ satisfying $h \circ f = \alpha h$. Due to the Julia Lemma \cite[Theorem 2.1.19]{Abate}, their convergence rate is at least exponential and in fact $\liminf_{n \to +\infty}(\abs{f^n(z)}/\alpha^n) > 0$ for all $z\in\H$. Therefore, it makes sense to wonder about all those self-maps $f$ where the convergence is exactly exponential of order $\alpha$ and the respective limit superior is finite, at least for one $z\in\H$. As it turns out (see Proposition \ref{prop:hyperbolic-rate-no-z}), the actual limit either exists in $\H$ for all $z\in\H$ and its value depends on $z$, or is infinite for all $z\in\H$ (see Remark \ref{rem:argument}). This dichotomy naturally leads to our first main definition. We say that $f$ is of \textit{extremal rate} if $\lim_{n \to +\infty}(\abs{f^n(z)}/\alpha^n) \in (0,+\infty)$, for all $z\in\H$.

If $f'(\infty) = 1$, $f$ is said to be \textit{parabolic}. Parabolic self-maps are further classified by their \textit{hyperbolic step}. Briefly, a parabolic self-map $f$ of $\H$ is said to have \textit{positive hyperbolic step} provided $\lim_{n\to+\infty}d_\H(f^{n+1}(z),f^{n}(z))>0$, for some --- and equivalently all due to \cite[Corollary 4.6.9.(i)]{Abate} --- $z\in\H$, where $d_\H$ denotes the hyperbolic distance in the upper half-plane. If the limit equals $0$ for some --- and equivalently all --- $z\in\H$, then $f$ has \textit{zero hyperbolic step}  (see Subsection \ref{subsec:discrete-dynamics} for more details). For parabolic maps of positive hyperbolic step, Pommerenke \cite{PommerenkeHalfPlane} established the tangential convergence of their orbits and the existence of a Koenigs function $h$ with $h \circ f = h+1$. We will prove that the rate is at least linear in this case and more specifically $\liminf_{n \to +\infty}(\abs{f^n(z)}/n) > 0$ for all $z\in\H$. Similarly to the hyperbolic setting, we will discover (see Proposition \ref{prop:phs}) that either the actual limit is real for all $z\in\H$ with a value independent of $z$, or is infinite for all $z\in\H$. This linear lower estimate for the rate of convergence is basically known to experts and may be derived from previous results concerning continuous semigroups of holomorphic functions \cite{DB-Rate-H,DB-Rate-P}. However, the overly descriptive nature of Proposition \ref{prop:phs} provides new insight. Once again, we are interested in the cases when the convergence is actually linear. Therefore, we say that parabolic maps of positive hyperbolic step are of \textit{extremal rate} if $\lim_{n \to +\infty}(\abs{f^n(z)}/n) \in (0,+\infty)$, for all $z\in\H$. 

It appears that a parabolic self-map of positive hyperbolic step is of extremal rate if and only if it is of \textit{finite shift} (see Proposition \ref{prop:phs}). Self-maps of the upper half-plane with finite shift have been thoroughly studied in recent years. In particular, the property of finite shift has been linked both to the Herglotz representation of the self-map \cite[Theorem 1.5]{CZ-FiniteShift} and the conformality of its Koenigs function at infinity \cite[Theorem 4.1]{CDP-C2}. Therefore, the extremal rate in this class of functions is related to several classical tools of Complex Analysis. This relation motivates our work to investigate similar results in the case of hyperbolic maps.

On the other hand, the situation for parabolic maps of zero hyperbolic step is more intricate, as orbits can exhibit non-trivial sets of slopes \cite[Theorem 2.13]{CCZRP-Slope}. Indeed, no universal lower estimate for the rate of convergence of orbits is known thus far. The lack of such an estimate does not allow the definition of an extremal rate in this class of functions. Hence, our study is devoted exclusively to hyperbolic maps and to a lesser extent to parabolic maps of positive hyperbolic step. The analogue research in the setting of continuous semigroups of holomorphic self-maps --- with special attention to the parabolic case of zero hyperbolic step --- will be covered in \cite{WIP}.

This paper aims to provide a comprehensive study of self-maps of extremal rate for these different classes of non-elliptic dynamics. After Sections \ref{sec:preliminaries} and \ref{sec:extremal rates}, where we give the required preliminaries to ease the exposition and provide the motivation for our work, we prove characterizations which certify the attainment of the extremal rate for self-maps using the following tools:

(i) Through the so-called \textit{Herglotz representation}, one can relate a self-map $f$ of the upper half-plane with a triplet $(\alpha,\beta,\mu)$, where $\alpha \geq 0$, $\beta \in \R$ and $\mu$ is a positive and finite measure on $\R$; see \eqref{eq:Herglotz}. In Section \ref{sec:Herglotz} we determine whether $f$ is of extremal rate in terms of this representation. This extends the characterizations found in \cite{CZ-FiniteShift} concerning self-maps of finite shift. More specifically, for hyperbolic self-maps we will prove the following (cf. Theorem \ref{thm:H-Herglotz}).
\begin{thm}
Let $f\vcentcolon \H \to \H$ be hyperbolic with Denjoy--Wolff point infinity and Herglotz representation 
$$f(z) = \alpha z + \beta + \int_{\R}\dfrac{1+tz}{t-z}d\mu(t), \qquad z \in \H.$$
Then, $f$ is of extremal rate if and only if
$$\int_{\R}\log(1+\abs{t})d\mu(t) < +\infty.$$
\end{thm}

(ii) In Section \ref{sec:conformality} we provide homologous characterizations in terms of the conformality at infinity of the Koenigs function. Recall that, in general, a function $h\vcentcolon\H\to\C$ satisfying $h(\infty)\vcentcolon=\angle\lim_{w\to\infty}h(w)=\infty$ is said to be \textit{conformal at infinity} provided $\angle\lim_{w\to\infty}(h(w)/w)\in\C\setminus\{0\}$. In our main theorem of this section, which expands a result in the case of hyperbolic continuous semigroups \cite[Theorem 4.2]{BCDM-Rate}, we will prove the following (cf. Theorem \ref{thm:conformality hyp}):
\begin{thm}
Let $f\vcentcolon\H\to\H$ be hyperbolic with Denjoy--Wolff point infinity and Koenigs function $h$. Then, $f$ is of extremal rate if and only if $h$ is conformal at infinity.
\end{thm}

(iii) Further characterizations are given in Section \ref{sec:behaviour}, where we investigate whether $f$ is of extremal rate in terms of its asymptotic behaviour near infinity. A similar idea first appeared in \cite{Pommerenke-Hyperbolic}. Using different techniques, we provide analogous results through Theorem \ref{thm:H-properties}.

(iv) We also explore a novel characterization based on the asymptotic behaviour of the hyperbolic distance between orbits and reference points in Section \ref{sec:hyp-dist}. For instance, in the parabolic case of positive hyperbolic step, we prove the following (cf. Theorem \ref{thm:hyperbolic rate phs}):
\begin{thm}
Let $f:\H\to\H$ be parabolic of positive hyperbolic step with Denjoy--Wolff point infinity. Then, $f$ is of extremal rate (equivalently, of finite shift) if and only if
\begin{equation*}
\lim\limits_{n\to+\infty}\left(d_\H(i,f^n(z))-\log (n)\right)\in\R, \quad \textup{for all }z\in\H.
\end{equation*}
\end{thm}
We also prove a hyperbolic analogue in Theorem \ref{thm:hyperbolic rate hyp}. Both statements build upon results found in \cite{Bracci}.

(v) Finally, in Section \ref{sec:comp-operators}, we translate our notion of extremality in the setting of the unit disc. Then, we apply our results in order to examine additional implications about the norms of the associated composition operators acting on classical spaces of analytic functions of the unit disc $\D$. The conclusions on this topic are contained in Corollaries \ref{thm:norm hyp} and \ref{thm:norm phs}.

\section{Preliminaries}
\label{sec:preliminaries}

\subsection{Hyperbolic distance}
Throughout the article we will heavily rely on hyperbolic geometry, and especially on the so-called \textit{hyperbolic distance}. For a concise presentation of its theory, we refer the interested reader to \cite[Chapter 5]{BCDM}. Since our study has the upper half-plane as its principal setting, we will solely work with the hyperbolic distance $d_\H$, which is given by the formula
\begin{equation}
\label{eq:hypdistance in H}
d_{\H}(z,w) = \dfrac{1}{2}\log\left(\dfrac{1+\rho_{\H}(z,w)}{1-\rho_{\H}(z,w)}\right), \quad \rho_{\H}(z,w) =\left| \dfrac{z-w}{z-\overline{w}}\right|, \quad z,w \in \H,
\end{equation}
where $\rho_\H$ is called the \textit{pseudo-hyperbolic distance} in $\H$.

\subsection{Discrete dynamics}
\label{subsec:discrete-dynamics}
A holomorphic self-map $f \vcentcolon \H \to \H$ is said to be \textit{non-elliptic} if $f(z) \neq z$ for all $z \in \H$. As stated in the Introduction, every such non-elliptic map has a Denjoy--Wolff point which is either infinity or some real number. From now on, we will always consider the Denjoy--Wolff point of a non-elliptic $f\vcentcolon\H\to\H$ to be infinity. In such a scenario, we have (up to a conjugation, see \cite[Corollary 2.5.5]{Abate})
$$f(\infty)\vcentcolon =\angle\lim_{z \to \infty}f(z) = \infty, \qquad f'(\infty) \vcentcolon = \angle\lim_{z \to \infty}\dfrac{f(z)}{z} \in [1,+\infty).$$
Given a point $z \in \H$, its \textit{orbit} under $f$ is the sequence $\{f^n(z)\}$, $n\in\N\cup\{0\}$, where $f^0 = \mathrm{Id_{\H}}$ and $f^{n} = f^{n-1} \circ f$, $n \in \N$. The main goal in Discrete Dynamics is to inspect the asymptotic behaviour of the orbits. A first observation is that, due to our assumption on the Denjoy--Wolff point, $\lim_{n \to +\infty}f^n(z) = \infty$. However, the dynamical behaviour of $\{f^n(z)\}$ greatly depends on $f$, as we are about to clarify.

Whenever $f'(\infty)>1$, the self-map $f$ is said to be \textit{hyperbolic}. In that case, Valiron \cite{Valiron} proved that the orbits of $f$ converge non-tangentially and with a definite angle. In other words, hyperbolicity ensures the existence in $(0,\pi)$ of the limit $\lim_{n \to +\infty}\arg(f^n(z))$ \cite[Theorem 4.3.4]{Abate}. Furthermore, disjoint orbits may land at infinity with different angles. In fact, for every $\theta \in (0,\pi)$ there exists $z \in \H$ such that $\lim_{n \to +\infty}\arg(f^n(z)) = \theta$ \cite[Property (2) (b)]{BracciCorradiniValiron}. In addition, Valiron proved that the dynamical behaviour of a hyperbolic map can be linearized in the following way:
\begin{theorem}
\label{thm:Valiron}
{\rm \cite{Valiron}}
Let $f \vcentcolon \H \to \H$ be hyperbolic with Denjoy--Wolff point infinity. Let $h_n \vcentcolon \H \to \H$, $n\in\N$, be given by
$$h_n(z) = \dfrac{f^n(z)}{\abs{f^n(i)}}, \qquad z \in \H.$$
Then, $\{h_n\}$ converges locally uniformly on $\H$ to a non-constant holomorphic map $h \vcentcolon \H \to \H$, as $n \to +\infty$. Moreover, $h$ satisfies $h \circ f = \alpha h$, where $\alpha = f'(\infty) > 1$.
\end{theorem}
The map $h$ appearing in the statement is known as the \textit{Koenigs function} of $f$.

On the other hand, if $f'(\infty)=1$, the map $f$ is said to be \textit{parabolic}. Parabolic self-maps are further classified with regard to their hyperbolic step. More precisely, a non-elliptic self-map $f$ is said to be of \textit{positive hyperbolic step} if
\begin{equation*}
\lim_{n \to +\infty}d_{\H}(f^{n+1}(z),f^n(z)) > 0 
\end{equation*}
for some (and equivalently all) $z \in \H$ (see \cite[Corollary 4.6.9]{Abate} for more information). Notice that the convergence of the limit is guaranteed by the Schwarz--Pick Lemma \cite[Corollary 1.1.16]{Abate}. If the limit is zero, then $f$ is said to be of \textit{zero hyperbolic step}.

Parabolic self-maps of positive hyperbolic step produce orbits that converge tangentially, that is $\lim_{n \to +\infty}\arg(f^n(z)) \in \{0,\pi\}$ for all $z\in\H$; see \cite[Remark 1]{PommerenkeHalfPlane}. On the contrary, the set of slopes can be non-trivial in the parabolic case of zero hyperbolic step \cite[Theorem 2.13]{CCZRP-Slope}.

Pommerenke, partly in collaboration with Baker \cite{BakerPommerenke,PommerenkeHalfPlane}, investigated maps that intertwine $f$ with a linear map. We summarize some of their results that we are going to need later on.
\begin{theorem}
\label{thm:Koenigs-P}
Let $f \vcentcolon \H \to \H$ be parabolic with Denjoy--Wolff point infinity.

{\rm (a)} {\normalfont \cite[Theorem 1]{PommerenkeHalfPlane}}
Fix $z_0\in\H$ and set $z_n=x_n+iy_n\vcentcolon =f^n(z_0)$. Then, the limit
$$b \vcentcolon = \lim_{n \to + \infty}\dfrac{x_{n+1}-x_n}{y_n}$$
exists in $\R$. Furthermore, $z_{n+1}/z_n \to 1$ and $y_{n+1}/y_n \to 1$, as $n \to + \infty$. Additionally, $b = 0$ if and only if $f$ is of zero hyperbolic step.

{\rm (b)} {\normalfont \cite[Theorem 4.6.8]{Abate}}
Then there exists a holomorphic map $h \vcentcolon \H \to \C$ such that $h \circ f = h + 1$.
\end{theorem}

If $f$ is of positive hyperbolic step, Pommerenke \cite[Theorem 1]{PommerenkeHalfPlane} constructed a solution of the equation $h \circ f = h + 1$ using the iterates of $f$. This map is known as the Koenigs map of $f$. Evidently, $h\circ f^n=h+n$, for all $n\in\N$.

Finally, non-elliptic self-maps may be further distinguished into two subclasses, regardless of their angular derivative at infinity or their hyperbolic step. More specifically, a non-elliptic self-map $f \vcentcolon \H \to \H$ with Denjoy--Wolff point infinity is said to be of \textit{finite shift} if $\lim_{n \to +\infty}\mathrm{Im}(f^n(z)) < +\infty$ for some (and hence all) $z \in \H$. Note that by the Julia Lemma the sequence $\{\mathrm{Im}(f^n(z))\}$ is non-decreasing and so its limit necessarily exists in $(0,+\infty]$; see \cite[Proposition 3.2]{CDP-C2}. If this limit is infinite, we say that $f$ is of \textit{infinite shift}. It may be proved that the class of non-elliptic self-maps with finite shift is a subclass of that of parabolic self-maps with positive hyperbolic step; see \cite[Proposition 3.3]{CDP-C2}.

\section{Extremal Rates}
\label{sec:extremal rates}

In this section we provide the groundwork for our work to follow, proving our first main results concerning the rate of divergence of $f^n$ to infinity, where $f\vcentcolon\H\to\H$ is holomorphic and non-elliptic with Denjoy--Wolff point infinity. Through these results, we will then proceed to the definition of \textit{extremal rate} for different classes of non-elliptic self-maps of the upper half-plane.

We commence with hyperbolic functions. As indicated in the Introduction, a consecutive application of the Julia Lemma \cite[Theorem 2.1.10]{Abate} assures that $\abs{f^n(z)} \geq \mathrm{Im}(z)\alpha^n$ for all $z \in \H$ and all $n \in \N$, where $\alpha \vcentcolon = f'(\infty) > 1$. That is, the orbit of a hyperbolic self-map converges to infinity (at least) exponentially fast. Even though the slowest possible rate is already known, we will prove an even stronger result which refines \cite[Lemma on p. 121]{Valiron}.

\begin{proposition}
\label{prop:hyperbolic-rate-no-z}
Let $f \vcentcolon \H \to \H$ be hyperbolic with Denjoy--Wolff point infinity, and set $\alpha = f'(\infty) > 1$. Then, for each $z \in \H$,
$$L(z) \vcentcolon = \lim_{n \to +\infty}\dfrac{f^n(z)}{\alpha^n} \in \H \cup \{\infty\}.$$
Moreover, if $L(z) = \infty$ for some $z \in \H$, then $L(z) = \infty$ for all $z \in \H$.
\begin{proof}
First of all, set $x_n \vcentcolon = \mathrm{Re}(f^n(i))$ and $y_n \vcentcolon = \mathrm{Im}(f^n(i))$, $n\in\N$. With this notation,
\begin{equation}\label{eq:L(z)}
\dfrac{f^n(i)}{\alpha^n} = \dfrac{x_n+iy_n}{\alpha^n} = \dfrac{y_n}{\alpha^n}\left(\dfrac{x_n}{y_n}+i\right).
\end{equation}
Since $f$ is hyperbolic, the orbit of $i$ converges to infinity with a definite angle in $(0,\pi)$. Ergo $\lim_{n \to +\infty} (x_n/y_n) \in \R$. Also, for each $n\in\N$,
$$\dfrac{y_{n+1}}{\alpha^{n+1}} = \dfrac{\alpha y_n + (y_{n+1}-\alpha y_n)}{\alpha^{n+1}} \geq \dfrac{y_n}{\alpha^n},$$
since $y_{n+1}-\alpha y_n \geq 0$ (recall that, due to the Julia Lemma, $z \mapsto f(z)-\alpha z$ is a self-map of $\H$). Therefore, the sequence $\{y_n/\alpha^n\}$ is increasing and its limit exists in $(0,+\infty]$. Combining the latter arguments with \eqref{eq:L(z)}, we conclude that $\lim_{n\to+\infty}(f^n(i)/\alpha^n) \in \H\cup\{\infty\}$. Setting $L(i)$ this limit, it is evident that $|L(i)|\in(0,+\infty]$. But, notice that for $z\in\H$
$$\dfrac{f^n(z)}{\alpha^n} = \dfrac{f^n(z)}{|f^n(i)|}\dfrac{|f^n(i)|}{\alpha^n}.$$
Taking limits as $n\to+\infty$, in view of Theorem \ref{thm:Valiron} we obtain
$$L(z)\vcentcolon =\lim_{n\to+\infty}\dfrac{f^n(z)}{\alpha^n}=h(z)|L(i)|\in\H\cup\{\infty\},$$
where $h$ is the Koenigs function of $f$, as defined in Theorem \ref{thm:Valiron}. It follows that $L(z)$ is infinite if and only if $L(i)$ is, something that implies the last part of our statement.
\end{proof}
\end{proposition}

\begin{remark}
\label{rem:argument}
In the proof, we saw that $L(z)=h(z)|L(i)|$. Consequently, whenever $L(i)\ne\infty$, the function $z\mapsto L(z)$ is non-constant and holomorphic due to the nature of the Koenigs map $h$. In addition, in such a case, the angle by which each orbit converges to infinity may be directly calculated through Proposition \ref{prop:hyperbolic-rate-no-z}. Indeed, $\lim_{n\to+\infty}\arg(f^n(z))=\arg(L(z))=\arg(h(z))$, for all $z\in\H$. This equality also follows from the semi-conformality of $h$ at infinity in the hyperbolic case; see \cite[Eq. (2.8) and Theorem 3]{PommerenkeHalfPlane} or \cite[pp. 14 and 15]{BracciCorradiniValiron}.
\end{remark}

In this article we are interested in those hyperbolic maps for which the slowest possible rate is attained. This is why, using Proposition \ref{prop:hyperbolic-rate-no-z}, we introduce the following well-defined notion.
\begin{definition}
\label{def:extremal hyp}
Let $f \vcentcolon \H \to \H$ be hyperbolic with Denjoy--Wolff point infinity, and set $\alpha = f'(\infty) > 1$. We say that $f$ is \textit{of extremal rate} whenever
$$\lim_{n \to +\infty}\dfrac{f^n(z)}{\alpha^n} \in \H$$
for some (and hence all) $z \in \H$.
\end{definition}

Leaving hyperbolic functions aside, we move on to the parabolic case. For holomorphic self-maps of $\H$ with Denjoy--Wolff point infinity that are parabolic of positive hyperbolic step, a universal lower bound for the divergence to infinity does not exist explicitly in the literature. By \cite[Remark 7.3]{BCDM-Rate}, we know that if $f$ is as above and in addition univalent, the rate is at least linear. In other words, for each $z\in\H$, there exists a constant $c$ depending on $z$ such that $|f^n(z)|\ge cn$, for all $n\in\N$. Observing the proof, it is easy to see that the assumption of univalence may be dropped and thus the rate is at least linear in general, something that is probably known to experts of the field. Nevertheless, our next result, which complements \cite[Corollary 1.6]{CZ-FiniteShift}, provides an even stronger grasp on the rhythm of the convergence.

\begin{proposition}
\label{prop:phs}
Let $f \vcentcolon \H \to \H$ be parabolic of positive hyperbolic step with Denjoy--Wolff point infinity. For each $z \in \H$ we have
$$L(z) \vcentcolon = \lim_{n \to +\infty}\dfrac{f^n(z)}{n} \in (\R \setminus \{0\}) \cup \{\infty\}$$
and the limit is independent of $z$. Moreover, $f$ is of finite shift if and only if $L(z) \in \R \setminus \{0\}$ for some (hence all) $z \in \H$.
\begin{proof}
First, assume that $f$ is of finite shift. Then, by \cite[Corollary 1.6]{CZ-FiniteShift}, there exists $C = C(f) \in \R \setminus \{0\}$, such that
$$\lim_{n \to +\infty}\dfrac{f^n(z)}{n} = C, \quad\text{for all }z\in\H.$$
Conversely, assume now that $f$ is of infinite shift and fix $z \in \H$. Set $z_n = f^n(z)$ and write $x_n = \mathrm{Re}(z_n)$ and $y_n = \mathrm{Im}(z_n)$, $n \in \N$. From Theorem \ref{thm:Koenigs-P}(a) we know that
$$b = \lim_{n \to +\infty}\dfrac{x_{n+1}-x_n}{y_n} \in \R \setminus \{0\},$$
due to the fact that $f$ is of positive hyperbolic step. Assume that $b > 0$ (the arguments for the other case follow in similar fashion). Since $\lim_{n \to +\infty}y_n = +\infty$, for every $C > 0$ we may find $N \in \N$ such that $x_{n+1}-x_n \geq C$ for all $n \geq N$. Applying this inequality inductively, we obtain $x_n \geq C(n-N)+x_N$, for all $n \geq N$. Therefore,
$$\dfrac{\abs{f^n(z)}}{n} \geq \dfrac{x_n}{n} \geq \ C \dfrac{n-N}{n}+\dfrac{x_N}{n}.$$
Taking limits as $n \to +\infty$ we infer that
$$\liminf_{n \to +\infty}\dfrac{\abs{f^n(z)}}{n} \geq C.$$
Since $C > 0$ is arbitrary, we conclude that
$$\lim_{n \to +\infty}\dfrac{\abs{f^n(z)}}{n} = +\infty.$$
Since every non-elliptic map $f\vcentcolon\H\to\H$ is either of finite shift or of infinite shift, this dichotomy leads to the stated equivalence and the independence of the limit $L(z)$ with respect to $z \in \H$. 
\end{proof}
\end{proposition}

Clearly, the last proposition allows for the following notion to be well-defined:
\begin{definition}
Let $f\vcentcolon\H\to \H$ be parabolic of positive hyperbolic step with Denjoy--Wolff point infinity. We say that $f$ is \textit{of extremal rate} whenever
\begin{equation*}
\lim\limits_{n\to+\infty}\frac{f^n(z)}{n}\in \R\setminus\{0\}
\end{equation*}
for some (and hence all) $z\in\H$.
\end{definition}
As a result, in the case of parabolic self-maps with positive hyperbolic step, Proposition \ref{prop:phs} clarifies that self-maps with an extremal rate of convergence towards the Denjoy--Wolff point coincide with those of finite shift.

There are several recent contributions dealing with this particular subclass of parabolic self-maps of positive hyperbolic step. In particular, parabolic self-maps of finite shift present concrete relations with their Herglotz representation \cite[Theorem 1.5]{CZ-FiniteShift} or the conformality of the respective Koenigs map at the Denjoy--Wolff point \cite[Theorem 4.1]{CDP-C2}. These connections serve as a motivation for our own research. More information about these topics are given in subsequent sections, where we also prove analogous results for hyperbolic self-maps of extremal rate.

\begin{remark}
It would be desirable to introduce a similar notion for parabolic self-maps of zero hyperbolic step. However, no universal lower bound with respect to $n\in\N$ is known for the quantity $\lvert f^n(z)\rvert$, $z\in\H$, in this case. A natural conjecture would be the satisfaction of $\liminf_{n\to+\infty}(\lvert f^n(z)\rvert/\sqrt{n})>0$, for all $z\in\H$. This thought is derived from the analogue result in the continuous setting \cite[Theorem 1.(c)]{DB-Rate-P}. Nevertheless, the techniques used in the proof of that result are difficult to adapt to the discrete setting since the image of the corresponding Koenigs map may have polar boundary \cite[Theorems 8.1 and 8.2]{CDMP-Zoo}.

The corresponding research for continuous semigroups of extremal rate will be conducted in \cite{WIP}.
\end{remark}

\section{Herglotz representation}
\label{sec:Herglotz}
We begin our examination by attempting to tie the extremal rate with the \textit{Herglotz representation} of holomorphic self-maps of the upper half-plane. A seminal result due to Herglotz \cite[Theorem 6.2.1]{Aaronson} certifies that every holomorphic function $f \vcentcolon \H \to \H$ can be uniquely written in the form
\begin{equation}
\label{eq:Herglotz}
f(z) = \alpha z + \beta + \int_{\R}\dfrac{1+tz}{t-z}d\mu(t), \quad z \in \H,
\end{equation}
where $\alpha \geq 0$, $\beta \in \R$ and $\mu$ is a positive finite measure on $\R$. In particular, some of the parameters in the representation above may be directly computed from the self-map. More specifically,
\begin{equation}
\label{eq:coefficients}
\alpha = \angle\lim_{z \to \infty}\dfrac{f(z)}{z} = f'(\infty), \qquad \beta = \mathrm{Re}(f(i)).
\end{equation}

With regard to our study, we see that $f$ is non-elliptic with Denjoy--Wolff point infinity if and only if $\alpha \geq 1$. In such a case, $f$ is hyperbolic if and only if $\alpha > 1$.

For the sake of brevity, we introduce the following notation.
\begin{definition}
We say that $f \vcentcolon \H \to \H$ is \textit{represented} by the triplet $(\alpha,\beta,\mu)$ if \eqref{eq:Herglotz} holds.
\end{definition}

In \cite[Theorem 1.5]{CZ-FiniteShift}, a relation between self-maps of finite shift (which coincide with the parabolic self-maps of positive hyperbolic step with an extremal rate of convergence; see Proposition \ref{prop:phs}) and the Herglotz representation was found. In this section we take these ideas to hyperbolic self-maps. We commence with some auxiliary lemmas.

\begin{lemma}
\label{lemma:non-tangential}
Let $a > 0$. There exists $C \vcentcolon = C(a) \ge 1 $ such that, for all $t \in \R$ and all $z \in \H$ with $\abs{\mathrm{Re}(z)} \leq a\mathrm{Im}(z)$, it holds that
\begin{equation*}
   \dfrac{1}{C}(t^2+\mathrm{Im}(z)^2) \leq \abs{t-z}^2 \leq C(t^2+\mathrm{Im}(z)^2). 
\end{equation*}
\begin{proof}
Fix $z\in\H$ with $|\mathrm{Re}(z)| \leq a \mathrm{Im}(z)$. For the sake of simplicity, set $x=\textup{Re}(z)$ and $y=\textup{Im}(z)$. We start with the right-hand side of the desired inequality. For $t \in \R$ we have
\begin{align}
\label{eq:non-tangential lemma, right-hand}
\abs{t-z}^2 & = (t-x)^2+y^2 \leq (|t|+|x|)^2+y^2 \\
& \leq  t^2+2a|t|y+a^2y^2+y^2 \leq  t^2+a(t^2+y^2)+(a^2+1)y^2 \notag \\
& \leq (a^2+a+1)(t^2+y^2), \notag
\end{align}
where we have consecutively made use of the triangle inequality and then of the trivial inequality $2|t|y\le t^2+y^2$.

We will now turn to the left-hand side of the intended outcome. Write
$$b = \dfrac{a^2+2-a\sqrt{a^2+4}}{2}.$$
It is easy to check that $0<b<1$ and that $\frac{1-b}{\sqrt{b}}=a$. But we know that $|x|<ay$ which implies $x^2<\frac{(1-b)^2}{b}y^2$. Through some straightforward rearranging, we are led to 
$$-\frac{bx^2}{1-b}+(1-b)y^2\ge0,$$
and adding a non-negative factor, we trivially deduce that
$$\left(\sqrt{1-b}t-\frac{x}{\sqrt{1-b}}\right)^2-\frac{bx^2}{1-b}+(1-b)y^2\ge0.$$
Doing certain simple computations, the last inequality directly leads to
\begin{equation}
\label{eq:non-tangential lemma, left-hand}
(t-x)^2+y^2\ge b(t^2+y^2).
\end{equation}
By combining relation \eqref{eq:non-tangential lemma, right-hand} with relation \eqref{eq:non-tangential lemma, left-hand} and by setting $C=\max\{b,a^2+a+1\}$, the result is obtained.
\end{proof}
\end{lemma}

\begin{lemma}
\label{lemma:estimate-sum}
Let $A > 0$, $\alpha > 1$, and $t \in \R \setminus \{0\}$. Then,
$$0 \leq \sum_{n = 0}^{+\infty}\dfrac{1}{t^2+A^2\alpha^{2n}} - \dfrac{\log(t^2+A^2)-\log(A^2)}{2t^2\log\alpha} \leq \dfrac{1}{t^2+A^2}.$$
\begin{proof}
Fix $t\in\R\setminus\{0\}$ and consider the real function $g:[0,+\infty)\to\R$ with $g(x)=1/(t^2+A^2\alpha^{2x})$. Since $\alpha>1$, the function $g$ is strictly decreasing in $[0,+\infty)$. Evidently, $g$ is also positive on $[0,+\infty)$. Therefore, we know that
\begin{equation}
\label{eq:integral lemma}
\int_{0}^{+\infty}g(x)dx \leq \sum\limits_{n=0}^{+\infty}g(n) \leq g(0)+\int_{0}^{+\infty}g(x)dx.
\end{equation}
We will explicitly calculate the integral in the relation above. Using the substitution $\alpha^{2x}=u$, we may write
\begin{equation*}
\int_{0}^{+\infty}g(x)dx = \frac{1}{2t^2\log \alpha}\int_{1}^{+\infty}\frac{\frac{t^2}{A^2}}{\frac{t^2}{A^2}+u}\frac{du}{u} = \frac{1}{2t^2\log \alpha}\log\left(\frac{t^2}{A^2}+1\right).
\end{equation*}
Keeping in mind that $g(0)=1/(t^2+A^2)$, inequality \eqref{eq:integral lemma} provides the desired result at once.
\end{proof}
\end{lemma}

With the last two lemmas in hand, we are ready for our main theorem of the section.

\begin{theorem}
\label{thm:H-Herglotz}
Let $f \vcentcolon \H \to \H$ be hyperbolic with Denjoy--Wolff point infinity and represented by the triplet $(\alpha,\beta,\mu)$. Then, $f$ is of extremal rate if and only if
$$\int_{\R}\log(1+\abs{t})d\mu(t) < +\infty.$$
\begin{proof}
Since $f$ is a priori hyperbolic, by \eqref{eq:coefficients} we deduce that $\alpha = f'(\infty) > 1$.

Choose $z_0 \in \H$ and let $z_n = f^n(z_0) = x_n + iy_n$, $n \in \N$, be its orbit under $f$.
Recall that $\lim_{n \to +\infty} (x_n/y_n) \in \R$ (see Subsection \ref{subsec:discrete-dynamics}) and thus $\lim_{n\to+\infty} (f^n(z_0)/y_n)\in \H$. Then, due to Definition \ref{def:extremal hyp}, $f$ is of extremal rate if and only if $\lim_{n \to +\infty}(y_n/\alpha^n) \in (0,+\infty)$. However, using \eqref{eq:Herglotz}, we notice that
$$\dfrac{y_n}{\alpha^ny_0} = \prod_{k = 0}^{n-1}\dfrac{y_{k+1}}{\alpha y_k} = \prod_{k = 0}^{n-1}\left(1+\dfrac{1}{\alpha}\int_{\R}\dfrac{1+t^2}{\abs{t-z_n}^2}d\mu(t)\right).$$
Therefore, correlating the product above with the respective sum, we deduce that $f$ is of extremal rate if and only if
\begin{equation}\label{eq:fubini}
\sum_{n = 0}^{+\infty}\int_{\R}\dfrac{1+t^2}{\abs{t-z_n}^2}d\mu(t) = \int_{\R}\left((1+t^2)\sum_{n = 0}^{+\infty}\dfrac{1}{\abs{t-z_n}^2}\right)d\mu(t)< +\infty,
\end{equation}
where we have used Fubini's Theorem. Moreover, by Lemma \ref{lemma:non-tangential}, relation \eqref{eq:fubini} is equivalent to
\begin{equation}\label{eq:fubini2}
\int_{\R}\left((1+t^2)\sum_{n = 0}^{+\infty}\dfrac{1}{t^2+y_n^2}\right)d\mu(t) < +\infty.
\end{equation}
Recall that using the upper half-plane version of the Julia Lemma, it is possible to find $C > 0$ so that
$$\alpha^ny_0 \leq y_n \leq C(\alpha+1)^n.$$
Therefore, it follows from Lemma \ref{lemma:estimate-sum} that there exists $K = K(y_0,\alpha) > 1$ with
$$\dfrac{1}{K} \dfrac{\log(1+\abs{t})}{t^2} \leq \sum_{n = 0}^{+\infty}\dfrac{1}{t^2+y_n^2} \leq K \dfrac{\log(1+\abs{t})}{t^2}, \qquad t \geq 1.$$
Returning back to \eqref{eq:fubini2}, the desired equivalence follows.
\end{proof}
\end{theorem}

Theorem \ref{thm:H-Herglotz} provides a characterization, with regard to the Herglotz representation, for whenever $|f^n|$ asymptotically behaves like $\alpha^n$, as $n\to+\infty$. Naturally, one may wonder whether this result may be generalized to connect the behaviour of $|f^n|$ with the corresponding Herglotz representation even when $\lim_{n\to+\infty}(|f^n(z)|/\alpha^n)=+\infty$, for all $z\in\H$. So, we aim to discover a measure of the divergence of the above limit. Towards this goal, we begin by developing some handy estimates.

\begin{lemma}
\label{lemma:normalization}

Consider the function
$$F(\alpha,n,y,t) = \log\left(\dfrac{\alpha^{2n}(t^2+y^2)}{t^2+\alpha^{2n}y^2}\right),$$
where $\alpha \in (1,+\infty)$, $n \in \N$, $y\in(0,+\infty)$, and $t \in \R$. For $\epsilon\in(0,+\infty)$, we have
\begin{equation}
\label{eq:normalization 0}
\dfrac{1}{\max\{y^2,1\}} \leq \dfrac{F(\alpha+\epsilon,n,y,t)}{F(\alpha,n,1,t)} \leq \dfrac{1}{\min\{y^2,1\}}\dfrac{\log(\alpha+\epsilon)}{\log(\alpha)}.
\end{equation}
\begin{proof}
Consider $G_1(\alpha,n,y,t,\epsilon) = F(\alpha+\epsilon,n,y,t)/F(\alpha,n,y,t)$. It can be readily checked that $\frac{\partial F(\alpha,n,y,t)}{\partial\alpha}=\frac{2nt^2}{\alpha(t^2+\alpha^{2n}y^2)}>0$ and thus $F$ is an increasing function of $\alpha\in(1,+\infty)$. Ergo $G_1\ge 1$. On the other hand, it can be checked that $G_1$ is an increasing function of $\abs{t}$. We omit the explicit calculations for the sake of increasing readability. Then,
\begin{equation}\label{eq:normalization 1}
G_1(\alpha,n,y,t,\epsilon) \leq \lim_{s \to \pm\infty}G_1(\alpha,n,y,s,\epsilon) = \dfrac{\log(\alpha+\epsilon)}{\log(\alpha)},
\end{equation}
for all $t\in\R$. Next, suppose that $y > 1$ and consider $G_2(\alpha,n,y,t) = F(\alpha,n,y,t)/F(\alpha,n,1,t)$. Once more, $G_2$ is an increasing function of $\abs{t}$. Since for each $y\in(0,+\infty)$
$$\lim_{t \to 0}\dfrac{F(\alpha,n,y,t)}{t^2}=\dfrac{\alpha^{2n}-1}{\alpha^{2n}y^2},$$
we conclude that
\begin{equation}\label{eq:normalization 2}
G_2(\alpha,n,y,t) \geq \lim_{s \to 0}G_2(\alpha,n,y,s)=\lim\limits_{s\to0}\left(\dfrac{F(\alpha,n,y,s)s^2}{F(\alpha,n,1,s)s^2}\right)=\dfrac{1}{y^2}.
\end{equation}
On the other side,
\begin{equation}\label{eq:normalization 3}
G_2(\alpha,n,y,t) \leq \lim_{s \to \pm\infty}G_2(\alpha,n,y,s)=1.
\end{equation}
Combining relations \eqref{eq:normalization 1}, \eqref{eq:normalization 2} and \eqref{eq:normalization 3} along with the fact that $G_1\ge1$ provides \eqref{eq:normalization 0}, for $y>1$. If $0 < y < 1$, a similar reasoning applies, although $G_2$ is now decreasing with respect to $\abs{t}$. Finally, for $y=1$ the result can be derived directly using only the results on $G_1$.
\end{proof}
\end{lemma}

We may now present a general estimate of $\abs{f^n(z)}/\alpha^n$ using the Herglotz representation.
\begin{proposition}
\label{prop:Herglotz divergence}
Let $f \vcentcolon \H \to \H$ be hyperbolic with Denjoy--Wolff point infinity and represented by the triplet $(\alpha,\beta,\mu)$. For each $z \in \H$ there exist two positive constants $C_1\vcentcolon =C_1(z)$ and $C_2\vcentcolon =C_2(z)$ such that
$$C_1 \leq \dfrac{\displaystyle\log\left(\frac{\abs{f^{n+1}(z)}}{\alpha^{n+1}}\right)}{\displaystyle\int_{\R}\frac{1+t^2}{t^2}\log\left(\frac{\alpha^{2n}(t^2+1)}{t^2+\alpha^{2n}}\right)d\mu(t)} \leq C_2, \qquad \textup{for all }n \in \N.$$
\begin{proof}
Fix $z\in\H$. First of all, since $f$ is hyperbolic, $\{f^n(z)\}$ converges non-tangentially to infinity and therefore there exists $K > 1$ such that $y_n \leq \abs{f^n(z)} \leq K y_n$, where as usual $y_n\vcentcolon =\mathrm{Im}(f^n(z))$, $n \in \N$, and $y_0=\mathrm{Im}(z)$. Using the Herglotz representation of $f$, we may write
\begin{equation}\label{eq:divergence 0}
\dfrac{y_{n+1}}{\alpha^{n+1}y_0} = \prod_{k = 0}^n \dfrac{y_{k+1}}{\alpha y_k} = \prod_{k = 0}^n\left(1+\dfrac{1}{\alpha}\int_{\R}\dfrac{1+t^2}{\abs{t-f^k(z)}^2}d\mu(t)\right).
\end{equation}
Therefore,
\begin{equation*}
\log\left(\dfrac{y_{n+1}}{\alpha^{n+1}y_0}\right) = \sum_{k = 0}^n\log\left(1+\dfrac{1}{\alpha}\int_{\R}\dfrac{1+t^2}{\abs{t-f^k(z)}^2}d\mu(t)\right).
\end{equation*}

To proceed with the upper bound, we notice that
\begin{align*}
\log\left(1+\dfrac{1}{\alpha}\int_{\R}\dfrac{1+t^2}{\abs{t-f^k(z)}^2}d\mu(t)\right) & \leq \dfrac{1}{\alpha}\int_{\R}\dfrac{1+t^2}{\abs{t-f^k(z)}^2}d\mu(t) \\
& \leq \dfrac{K_1}{\alpha}\int_{\R}\dfrac{1+t^2}{t^2+y_k^2}d\mu(t) \\
& \leq \dfrac{K_1}{\alpha}\int_{\R}\dfrac{1+t^2}{t^2+\alpha^{2k}y_0^2}d\mu(t),
\end{align*}
where we have used Lemma \ref{lemma:non-tangential} and the Julia Lemma. This implies that
\begin{equation}
\label{eq:upper-estimate-1}
\log\left(\dfrac{y_{n+1}}{\alpha^{n+1}y_0}\right) \leq \dfrac{K_1}{\alpha}\int_{\R}\sum_{k = 0}^n\dfrac{1+t^2}{t^2+\alpha^{2k}y_0^2}d\mu(t).
\end{equation}
Arguing as in the proof of Lemma \ref{lemma:estimate-sum}, we can see that
\begin{align*}
\sum_{k = 0}^n\dfrac{1}{t^2+\alpha^{2k}y_0^2} & \leq \dfrac{1}{t^2+y_0^2}+\int_0^n\dfrac{dx}{t^2+\alpha^{2x}y_0^2} \\
& = \dfrac{1}{t^2+y_0^2}+\dfrac{\log\left(\dfrac{\alpha^{2n}(t^2+y_0^2)}{t^2+\alpha^{2n}y_0^2}\right)}{2t^2\log(\alpha)}.
\end{align*}
Using derivatives one can verify that
$$t \mapsto (t^2+y_0^2)\dfrac{\log\left(\dfrac{\alpha^{2n}(t^2+y_0^2)}{t^2+\alpha^{2n}y_0^2}\right)}{2t^2\log(\alpha)}$$
is an increasing function of $\abs{t}$, and 
$$\lim_{t \to 0}(t^2+y_0^2)\dfrac{\log\left(\dfrac{\alpha^{2n}(t^2+y_0^2)}{t^2+\alpha^{2n}y_0^2}\right)}{2t^2\log(\alpha)} = \dfrac{\alpha^{2n}-1}{2\alpha^{2n}\log(\alpha)} \geq \dfrac{\alpha^2-1}{2\alpha^2\log(\alpha)}, \qquad n \geq 1.$$
Thus, there exists $K_2 > 0$, not depending on $n$, such that
$$\dfrac{1}{t^2+y_0^2} \leq K_2\dfrac{\log\left(\dfrac{\alpha^{2n}(t^2+y_0^2)}{t^2+\alpha^{2n}y_0^2}\right)}{2t^2\log(\alpha)}.$$
This means that we can find some $K_3 > 0$ such that
$$\sum_{k = 0}^n\dfrac{1}{t^2+\alpha^{2k}y_0^2} \leq K_3\dfrac{\log(\alpha^{2n}(t^2+y_0^2)/(t^2+\alpha^{2n}y_0^2))}{t^2}, \qquad n \geq 1.$$
Using the latter relation in \eqref{eq:upper-estimate-1}, we prove that
$$\log\left(\dfrac{y_{n+1}}{\alpha^{n+1}y_0}\right) \leq K_4\int_{\R}\dfrac{1+t^2}{t^2}\log\left(\frac{\alpha^{2n}(t^2+1)}{t^2+\alpha^{2n}}\right)d\mu(t),$$
for some $K_4 > 0$, from which the upper estimate in the statement follows.

To address the lower estimate, notice that
\begin{align*}
\int_{\R}\dfrac{1+t^2}{\abs{t-f^k(z)}^2}d\mu(t) & \leq K_5\int_{\R}\dfrac{1+t^2}{t^2+y_k^2}d\mu(t) \leq K_5\int_{\R}\dfrac{1+t^2}{t^2+y_0^2}d\mu(t) \\
& \leq K_5\left(1+\dfrac{1}{y_0^2}\right)\mu(\R) < +\infty,   
\end{align*}
where $K_5 > 0$ is obtained from Lemma \ref{lemma:non-tangential}. Therefore, there exists $K_6 > 0$ so that
\begin{equation}
\label{eq:lower-estimate-1}
\log\left(1+\dfrac{1}{\alpha}\int_{\R}\dfrac{1+t^2}{\abs{t-f^k(z)}^2}d\mu(t)\right) \geq K_6 \int_{\R}\dfrac{1+t^2}{\abs{t-f^k(z)}^2}d\mu(t).
\end{equation}

Now, fix $\epsilon > 0$. By the Julia Lemma there exists $Y > 0$ so that $y_k \leq Y(\alpha+\epsilon)^k$, for all $k\in\N$. Then, using also Lemma \ref{lemma:non-tangential}, we get from \eqref{eq:divergence 0} that
\begin{align}
\label{eq:divergence 6}
\log\left(\dfrac{y^{n+1}}{\alpha_{n+1}y_0}\right) & \geq K_6\int_{\R}\sum_{k=0}^n\dfrac{1+t^2}{\abs{t-f^k(z)}^2}d\mu(t) \\
& \geq K_7\int_{\R}\sum_{k = 0}^n\dfrac{1}{t^2+y_k^2}d\mu(t) \notag \\
& \geq K_7\int_{\R}\sum_{k = 0}^n\dfrac{1}{t^2+(\alpha+\epsilon)^{2k}Y^2}d\mu(t), \notag
\end{align}
for some $K_7 > 0$. Arguing as in the proof of Lemma \ref{lemma:estimate-sum}, we have that
\begin{align*}
\sum_{k = 0}^n\dfrac{1}{t^2+(\alpha+\epsilon)^{2k}Y^2} & \geq \int_0^n\dfrac{dx}{t^2+(\alpha+\epsilon)^{2x}Y^2}  = \dfrac{\log\left(\dfrac{(\alpha+\epsilon)^{2n}(t^2+Y^2)}{t^2+(\alpha+\epsilon)^{2n}Y^2}\right)}{2t^2\log(\alpha+\epsilon)} \\
& \geq K_8\dfrac{\log\left(\dfrac{(\alpha^{2n}(t^2+1)}{t^2+a^{2n}}\right)}{t^2},
\end{align*}
where the last inequality follows from Lemma \ref{lemma:normalization} at once. As a result, in combination with \eqref{eq:divergence 6}, there exists $K_9 > 0$ with
$$\log\left(\dfrac{y_{n+1}}{\alpha^{n+1}y_0}\right) \geq K_9\int_{\R}\dfrac{1+t^2}{t^2}\log\left(\frac{\alpha^{2n}(t^2+1)}{t^2+\alpha^{2n}}\right)d\mu(t),$$
which proves the result.
\end{proof}
\end{proposition}
\begin{remark}
Notice that
$$\lim_{n \to +\infty} \log\left(\frac{\alpha^{2n}(t^2+1)}{t^2+\alpha^{2n}}\right) = \log(1+t^2),$$
and therefore Proposition \ref{prop:Herglotz divergence} provides a clear and qualitative generalization of Theorem \ref{thm:H-Herglotz}.
\end{remark}

\section{Conformality}
\label{sec:conformality}
The behaviour of a holomorphic map of the upper half-plane $f \vcentcolon \H \to \C$ at a boundary point $\xi \in \partial_{\infty}\H = \R \cup \{\infty\}$ is a classical and well-established topic of research in Geometric Function Theory. In this section we will mainly use the notion of conformality of a map at a boundary point (for instance, as described in \cite[Section 4.3]{PommerenkeConfMaps}), which has appeared before in the context of Complex Dynamics; see \cite{BCDM-Rate,CDP-C2,GKMR,GKR}. More specifically, we will find a relation between self-maps of extremal rate and the conformality of their Koenigs function (see Theorems \ref{thm:Valiron} and \ref{thm:Koenigs-P}(b)) at the Denjoy--Wolff point.

As always, we will focus in the case of non-elliptic self-maps whose Denjoy--Wolff point is infinity. Therefore, infinity is a boundary fixed point for the Koenigs functions of such self-maps. Due to this, for the purposes of our article it suffices to introduce the notion of conformality at infinity for maps satisfying
$$f(\infty) = \angle\lim_{z \to \infty}f(z) = \infty.$$
\begin{definition}
Let $f \vcentcolon \H \to \C$ be a holomorphic map with $f(\infty) = \infty$. We say that $f$ is \textit{conformal at infinity} if
\begin{equation}
\label{eq:conformality-H}
f'(\infty) = \angle\lim_{z \to \infty}\dfrac{f(z)}{z} \in \C \setminus \{0\}.
\end{equation}
\end{definition}

Having mentioned the necessary definition, we are ready to proceed to the main result of the section. We will inspect the relation between the rate of convergence to infinity for self-maps of the upper half-plane and the conformality at infinity for the corresponding Koenigs function. Our examination is inspired by the already-known relation between the two notions in the case of parabolic self-maps of positive hyperbolic step. Recall that in \cite[Theorem 4.1]{CDP-C2}, the authors prove that such a self-map is of finite shift if and only if its Koenigs function is conformal at infinity. But as we mentioned before, functions of finite shift are exactly the parabolic self-maps of positive hyperbolic step that attain the extremal rate (see Proposition \ref{prop:phs}). 

In the following result, which provides a discrete analogue of \cite[Theorem 4.2]{BCDM-Rate}, we notice that the equivalence described above remains true for hyperbolic self-maps.

\begin{theorem}
\label{thm:conformality hyp}
Let $f \vcentcolon \H \to \H$ be hyperbolic with Denjoy--Wolff point infinity, and let $h \vcentcolon \H \to \H$ be its Koenigs function. Let $\alpha=f'(\infty)>1$. Then, $f$ is of extremal rate if and only if $h$ is conformal at infinity.

Moreover, in such a case, we have that
\begin{equation}\label{eq:conformality hyp limit}
\lim_{n \to +\infty}\dfrac{f^n(z)}{\alpha^n} = h(z)\left(\angle\lim_{w \to \infty}\dfrac{h(w)}{w}\right)^{-1}, \quad\textup{for all }z\in\H.
\end{equation}
\begin{proof}
As a consequence of the Julia--Wolff--Carath\'eodory Theorem (see \cite[Theorem 1.7.8]{BCDM}) we know that $h$ is conformal at infinity if and only if
\begin{equation}
\label{eq:conformality}
\inf_{z \in \H}\dfrac{\mathrm{Im}(h(z))}{\mathrm{Im}(z)} > 0.
\end{equation}
Fix $z \in \H$. For each $n \in \N$, using the sequence of mappings introduced in Theorem \ref{thm:Valiron}, we have that
$$\dfrac{\mathrm{Im}(h_n(z))}{\mathrm{Im}(z)} = \dfrac{1}{\mathrm{Im}(z)}\mathrm{Im}\left(\dfrac{f^n(z)}{\abs{f^n(i)}}\right) = \dfrac{\mathrm{Im}(f^n(z))}{\mathrm{Im}(z)}\dfrac{1}{\abs{f^n(i)}} \geq \dfrac{\alpha^n}{\abs{f^n(i)}},$$
where we have made an inductive use of the Julia Lemma, that is
$$\mathrm{Im}(f(z)) \geq \alpha \mathrm{Im}(z), \qquad z \in \H.$$
If $f$ is of extremal rate, we know that
$$\lim_{n \to +\infty}\dfrac{\alpha^n}{\abs{f^n(i)}} \in (0,+\infty).$$
Therefore, we can find $C = C(f) > 0$ and $N = N(f) \in \N$ such that
$$\dfrac{\mathrm{Im}(h_n(z))}{\mathrm{Im}(z)} \geq C, \qquad n \geq N.$$
But recall that $\{h_n\}$ converges uniformly on compacta to $h$. So, letting $n \to +\infty$, we conclude that
$$\dfrac{\mathrm{Im}(h(z))}{\mathrm{Im}(z)} \geq C.$$
Since $C$ is independent of $z$, this means that $\inf_{z \in \H}(\mathrm{Im}(h(z))/\mathrm{Im}(z)) > 0$. By \eqref{eq:conformality}, we conclude that $h$ is conformal at infinity.

For the converse direction, assume that
\begin{equation}\label{eq:conformality hyp 1}
\lim_{n \to +\infty}\dfrac{\abs{f^n(z)}}{\alpha^n} = +\infty
\end{equation}
for some $z \in \H$. Since $f$ is hyperbolic, the orbit of $z$ converges to infinity non-tangentially and with a definite angle (see Subsection \ref{subsec:discrete-dynamics}), that is $\lim_{n\to+\infty}\arg(f^n(z))\in(0,\pi)$, and hence
\begin{equation}
\label{eq:conformality hyp 2}
\lim\limits_{n\to+\infty}\frac{|f^n(z)|}{\mathrm{Im}(f^n(z))}\in(0,+\infty).
\end{equation}
In that case, combining relations \eqref{eq:conformality hyp 1} and \eqref{eq:conformality hyp 2}, we have
$$\lim_{n \to +\infty}\dfrac{\mathrm{Im}(h(f^n(z)))}{\mathrm{Im}(f^n(z))} = \lim_{n \to +\infty}\dfrac{\alpha^n\mathrm{Im}(h(z))}{\mathrm{Im}(f^n(z))} = \lim_{n \to +\infty}\dfrac{\alpha^n\mathrm{Im}(h(z))|f^n(z)|}{\abs{f^n(z)}\mathrm{Im}(f^n(z))} = 0,$$
since by the definition of the Koenigs function $h(f^n(z))=\alpha^nh(z)$; see Theorem \ref{thm:Valiron}. As a consequence, $\inf_{z \in \H}(\mathrm{Im}(h(z))/\mathrm{Im}(z)) = 0$. Then, using \eqref{eq:conformality}, we see that $h$ is not conformal at infinity.

Finally, if any (and hence both) of the conditions holds, \eqref{eq:conformality hyp limit} follows at once from the equality
\begin{align*}
\lim_{n \to +\infty}\dfrac{f^n(z)}{\alpha^n} & = \lim_{n \to +\infty}\left(\dfrac{f^n(z)}{h(f^n(z))}\dfrac{h(f^n(z))}{\alpha^n}\right) \\
& = \lim_{n \to +\infty}\dfrac{f^n(z)}{h(f^n(z))}h(z) = \left(\angle\lim_{w \to \infty}\dfrac{h(w)}{w}\right)^{-1}h(z),
\end{align*}
where we have used Theorem \ref{thm:Valiron} and the fact that $f^n(z)$ converges non-tangentially to infinity.
\end{proof}
\end{theorem}

\section{Asymptotic behaviour near infinity}
\label{sec:behaviour}
Let $f \vcentcolon \H \to \H$ be a non-elliptic self-map with Denjoy--Wolff point infinity. It would also be useful to be able to determine whether $f$ is of extremal rate or not directly from the properties of $f$ and not other related aspects. This will be the objective of the present section. To this extent, we commence by providing the following helpful result.

\begin{lemma}
\label{lemma:equivalence-integral-t}
Let $f \vcentcolon \H \to \H$ be holomorphic and represented by the triplet $(\alpha,\beta,\mu)$. The following are equivalent:
\begin{enumerate}[\hspace{0.5cm}\rm(a)]
\item $\displaystyle\int_{\R}\abs{t}d\mu(t) < +\infty$.
\item $\displaystyle\int_1^{+\infty}\dfrac{\mathrm{Im}(f(iy))-\alpha y}{y}dy < +\infty$.
\end{enumerate}
\begin{proof}
First of all, by the Julia Lemma, notice that the integrand in (b) is non-negative. For the desired equivalence, we use the Herglotz representation of $f$ and then Fubini's Theorem to see that
\begin{align}\label{eq:asymptotic 1}
\notag\int_1^{+\infty}\dfrac{\mathrm{Im}(f(iy))-\alpha y}{y}dy &=  \int_{\R}\int_1^{+\infty}\dfrac{1+t^2}{t^2+y^2}dyd\mu(t) \\
& = \int_{\R \setminus \{0\}}\dfrac{(1+t^2)\arctan(t)}{t}d\mu(t) + \mu(\{0\}).
\end{align}
Consider $F$ to be the function given by the formula
$$F(t) = \dfrac{(1+t^2)\arctan(t)}{t}, \qquad t \neq 0.$$
Notice that
$$\lim_{t \to 0}F(t) = 1, \qquad \lim_{t \to \pm\infty}\dfrac{F(t)}{\abs{t}} = \dfrac{\pi}{2}.$$
Therefore, the integrand in the latter integral in \eqref{eq:asymptotic 1} is bounded near $0$, while it is comparable to $|t|$ far from $0$. Evidently, the equivalence in the statement follows.
\end{proof}
\end{lemma}
Condition (a) in the preceding lemma has appeared before in reference to the Herglotz representation of parabolic self-maps and their dynamical properties. For instance, it appears in a characterization of the hyperbolic step of parabolic self-maps \cite[Section 3]{HS} and in relation to self-maps of finite shift \cite[Theorem 1.5]{CZ-FiniteShift}. We now use it to extract the following result.
\begin{proposition}
Let $f \vcentcolon \H \to \H$ be parabolic of positive hyperbolic step with Denjoy--Wolff infinity. If $f$ is of extremal rate, then
$$\int_1^{+\infty}\dfrac{\mathrm{Im}(f(iy))-y}{y}dy<+\infty.$$
\begin{proof}
It follows from the combination of Lemma \ref{lemma:equivalence-integral-t} and \cite[Theorem 1.5]{CZ-FiniteShift}.
\end{proof}
\end{proposition}

In the hyperbolic case we can obtain a full characterization.
\begin{theorem}
\label{thm:H-properties}
Let $f \vcentcolon \H \to \H$ be hyperbolic with Denjoy--Wolff point infinity. Let $\alpha = f'(\infty) > 1$. The following are equivalent:
\begin{enumerate}[\hspace{0.5cm}\rm(a)]
\item $f$ is of extremal rate.
\item $\displaystyle\int_1^{+\infty}\abs{\dfrac{f(iy)-i\alpha y}{y^2}}dy < +\infty$.
\item $\displaystyle\int_1^{+\infty}\dfrac{\mathrm{Im}(f(iy))-\alpha y}{y^2}dy < +\infty$.
\end{enumerate}
\begin{proof}
 Let $y > 0$. We first notice that
\begin{align*}
\mathrm{Re}\left(\dfrac{f(iy)-i\alpha y}{y^2}\right) & = \dfrac{\beta}{y^2}+\int_{\R}\dfrac{t}{y^2}\dfrac{1-y^2}{t^2+y^2}d\mu(t), \\
\mathrm{Im}\left(\dfrac{f(iy)-i\alpha y}{y^2}\right) & = \int_{\R}\dfrac{1}{y}\dfrac{1+t^2}{t^2+y^2}d\mu(t).
\end{align*}

Using these expressions, we have that
$$\abs{\mathrm{Re}\left(\dfrac{f(iy)-i\alpha y}{y^2}\right)} \leq \dfrac{\abs{\beta}}{y^2}+\int_{\R}\dfrac{\abs{t}\abs{1-y^2}}{y^2(t^2+y^2)}d\mu(t).$$
Then, using Fubini's Theorem, we obtain
\begin{align}
\int_1^{+\infty}\abs{\mathrm{Re}\left(\dfrac{f(iy)-i\alpha y}{y^2}\right)}dy & \leq \int_1^{+\infty}\dfrac{\abs{\beta}}{y^2}dy + \int_{\R}\left(\int_1^{+\infty}\dfrac{\abs{t}\abs{1-y^2}}{y^2(t^2+y^2)}dy\right)d\mu(t) \nonumber \\
& \leq \abs{\beta}+\int_{\R}\left(\abs{t}\int_1^{+\infty}\dfrac{dy}{t^2+y^2}\right)d\mu(t) \nonumber \\
& \leq \abs{\beta}+\int_{\R}\abs{\arctan(t)}d\mu(t) < +\infty, \label{eq:Re}
\end{align}
since $\mu$ is a positive and finite measure. Therefore, the integral of the real part is always finite due to \eqref{eq:Re}, and hence (b) and (c) are clearly equivalent.

To continue with the imaginary part, for $t \neq 0$, executing integration by parts, we notice that
\begin{equation}\label{eq:integral}
\int_{1}^{+\infty}\dfrac{dy}{y(t^2+y^2)} = \lim\limits_{y\to+\infty}\left(\dfrac{\log(y)}{t^2}-\dfrac{\log(t^2+y^2)}{2t^2}\right)+\dfrac{\log(1+t^2)}{2t^2}.
\end{equation}
Then, using Fubini's Theorem and finding that the limit in \eqref{eq:integral} is equal to $0$, we get
\begin{align}
\int_1^{+\infty}\mathrm{Im}\left(\dfrac{f(iy)-i\alpha y}{y^2}\right)dy & = \int_{\R}\left(\int_1^{+\infty}\dfrac{1}{y}\dfrac{1+t^2}{t^2+y^2}dy\right)d\mu(t) \nonumber \\
& = \int_{\R \setminus \{0\}}(1+t^2)\dfrac{\log(1+t^2)}{2t^2}d\mu(t) + \dfrac{1}{2}\mu(\{0\}). \label{eq:Im}
\end{align}
Setting
$$F(t) = (1+t^2)\dfrac{\log(1+t^2)}{2t^2}, \qquad t \ne 0,$$
we see that
\begin{equation}
\label{eq:equivalences-Im}
\lim_{t \to 0}F(t) = \dfrac{1}{2}, \qquad \lim_{t \to \pm\infty}\dfrac{F(t)}{\log(1+\abs{t})} = 1.    
\end{equation}
Therefore, arguing as at the end of the proof for Lemma \ref{lemma:equivalence-integral-t}, the desired equivalence for the extremal rate is established by using Theorem \ref{thm:H-Herglotz}.
\end{proof}
\end{theorem}
\begin{remark}
In our previous theorem, the equivalence between (a) and (b) was first established by Pommerenke; see \cite[Theorem 2]{Pommerenke-Hyperbolic}. His proof relies on a previous characterization due to Valiron; see \cite[Th\'eor\`eme IV]{Valiron}.
\end{remark}

\section{Hyperbolic distance}
\label{sec:hyp-dist}
In this section we will examine how the asymptotic behaviour of the hyperbolic distance between the orbit and a reference point in $\H$ determines the appearance of the extremal rate. We will work in the setting of both hyperbolic maps and parabolic maps of positive hyperbolic step. With the use of hyperbolic distance we will find characterizations of the extremal rates, thus finding an equivalent definition for finite shift as well.

We start with hyperbolic self-maps. As we mentioned before, given a hyperbolic function $f\vcentcolon\H\to\H$, each orbit converges to the Denjoy--Wolff point with a definite angle and in particular non-tangentially. However, distinct orbits land with different angles. Assuming that $f$ has Denjoy--Wolff point infinity allows us to express these angles easily. From now on, for a hyperbolic function $f$ and for $z\in\H$, we set
$$\lim\limits_{n\to+\infty}\arg(f^n(z))=:\theta(z)\in(0,\pi).$$
Clearly, if $z_1,z_2\in\H$ belong to the same orbit, then $\theta(z_1)=\theta(z_2)$.

With all the above in mind, we proceed to our result:
\begin{theorem}
\label{thm:hyperbolic rate hyp}
Let $f\vcentcolon\H\to\H$ be hyperbolic with Denjoy--Wolff point infinity. Let $\alpha=f'(\infty)>1$. Then, $f$ is of extremal rate if and only if the limit
$$\lim\limits_{n\to+\infty}\left(d_\H(i,f^n(z))-\frac{\log(\alpha)}{2}n\right)$$
exists in $\R$, for all $z\in\H$. Moreover, if any of the conditions are met, then
\begin{equation}\label{eq:hyp hyp 1}
\lim\limits_{n\to+\infty}\left(d_\H(i,f^n(z))-\frac{\log(\alpha)}{2}n\right)=\frac{1}{2}\log\left(\dfrac{|h(z)L(i)|}{\sin(\arg(h(z)))}\right),
\end{equation}
where $h$ is the Koenigs function of $f$ and $L$ is the function defined in Proposition \ref{prop:hyperbolic-rate-no-z}.
\end{theorem}

\begin{proof}
Recall that by \eqref{eq:hypdistance in H}
\begin{equation}\label{eq:hyp hyp 2}
d_{\H}(i,f^n(z)) = \dfrac{1}{2}\log\dfrac{1+\abs{\frac{f^n(z)-i}{f^n(z)+i}}}{1-\abs{\frac{f^n(z)-i}{f^n(z)+i}}} = \dfrac{1}{2}\log\dfrac{\left(1+\abs{\frac{f^n(z)-i}{f^n(z)+i}}\right)^2}{1-\abs{\frac{f^n(z)-i}{f^n(z)+i}}^2}.
\end{equation}
Since $f^n(z)$ converges to infinity, as $n \to +\infty$, we have that
\begin{equation}\label{eq:hyp hyp 3}
\lim_{n \to +\infty}\left(1+\abs{\frac{f^n(z)-i}{f^n(z)+i}}\right)^2 = 4.
\end{equation}
For the denominator of the fraction in the latter logarithm in \eqref{eq:hyp hyp 2}, for each $n\in\N$ let $f^n(z) = x_n+iy_n$, and notice that
$$1-\abs{\dfrac{f^n(z)-i}{f^n(z)+i}}^2 = \dfrac{\abs{f^n(z)+i}^2-\abs{f^n(z)-i}^2}{\abs{f^n(z)+i}^2} = \dfrac{4y_n}{\abs{f^n(z)+i}^2}.$$
Therefore, using relations \eqref{eq:hyp hyp 2} and \eqref{eq:hyp hyp 3}, we conclude that
$$\lim_{n \to +\infty}\left[d_{\H}(i,f^n(z)) - \dfrac{1}{2}\log\left(\dfrac{\abs{f^n(z)+i}^2}{y_n}\right)\right] = 0.$$
On another note, since the convergence to infinity is realized through a definite angle, we have
$$\lim_{n \to +\infty}\dfrac{\abs{f^n(z)+i}}{y_n} = \sqrt{\dfrac{1}{\tan^2(\theta(z))}+1} = \dfrac{1}{\sin(\theta(z))}.$$
Summing up,
\begin{equation}\label{eq:hyp hyp 4}
\lim_{n \to +\infty}\left(d_{\H}(i,f^n(z))-\dfrac{1}{2}\log(\abs{f^n(z)+i})\right) = \dfrac{1}{2}\log\left(\dfrac{1}{\sin(\theta(z))}\right).
\end{equation}
But notice that by definition $f$ is of extremal rate if and only if
\begin{equation}\label{eq:hyp hyp 5}
\lim_{n \to +\infty}\left(\log(\abs{f^n(z)+i})-n\log(\alpha)\right) = \lim_{n \to +\infty}\log\left(\dfrac{\abs{f^n(z)+i}}{\alpha^n}\right) \in \R.
\end{equation}
Therefore, combining \eqref{eq:hyp hyp 4} with \eqref{eq:hyp hyp 5}, we deduce that $f$ is of extremal if and only if
$$\lim_{n \to \infty}\left(d_{\H}(i,f^n(z))-\dfrac{\log(\alpha)}{2}n\right) \in \R.$$
Again from \eqref{eq:hyp hyp 4} and \eqref{eq:hyp hyp 5}, we see that the value of the limit is actually equal to $\frac{1}{2}\log(\frac{|L(z)|}{\sin(\theta(z))})$. But when the rate is extremal and hence $L(i)\ne\infty$, from Remark \ref{rem:argument}, we have $L(z)=h(z)|L(i)|$, while $\theta(z)=\arg(h(z))$. These observations lead to \eqref{eq:hyp hyp 1}.
\end{proof}

\begin{remark}
If $f$ is not of extremal rate, combining \eqref{eq:hyp hyp 4}, \eqref{eq:hyp hyp 5}, and Proposition \ref{prop:hyperbolic-rate-no-z}, we obtain
$$\lim\limits_{n\to+\infty}\left(d_\H(i,f^n(z))-\frac{\log(\alpha)}{2}n\right) = +\infty.$$
In such a case, $\abs{L(i)} = + \infty$ and so \eqref{eq:hyp hyp 1} still holds.
\end{remark}

Moving on, we will examine parabolic self-maps of positive hyperbolic step with extremal rate, or in other words non-elliptic self-maps of finite shift. First, we will need the following lemma (cf. \cite[Theorem 1.7.8]{BCDM}) relating conformality at the boundary with hyperbolic distance.

\begin{lemma}
\label{lm:conformality hyperbolic distance}
Let $h\vcentcolon\H\to\H$ be holomorphic. Then, $h$ is conformal at infinity if and only if 
\begin{equation*}
\liminf\limits_{\H\ni w\to\infty}\left(d_\H(i,w)-d_\H(i,h(w))\right)<+\infty.
\end{equation*}
\end{lemma}
Applying this result on the Koenigs function of a parabolic self-map of positive hyperbolic step will aid us in proving the following theorem which is the counterpart of Theorem \ref{thm:hyperbolic rate hyp} and refines \cite[Theorem 1.2(ii)]{KTZ}. Before that, recall that whenever a function $f\vcentcolon\H\to\H$ with Denjoy--Wolff point infinity is of finite shift, the imaginary part of each orbit is bounded. In such a case, we set $I(z)\vcentcolon =\lim_{n\to+\infty}\mathrm{Im}(f^n(z))$, $z\in\H$.

\begin{theorem}
\label{thm:hyperbolic rate phs}
Let $f:\H\to\H$ be parabolic of positive hyperbolic step with Denjoy--Wolff point infinity. Then, $f$ is of extremal rate if and only if the limit
\begin{equation*}
\lim\limits_{n\to+\infty}\left(d_\H(i,f^n(z))-\log (n)\right)
\end{equation*}
exists in $\R$, for all $z\in\H$. Moreover, if any of the conditions are met, then
\begin{equation}\label{eq:hyperbolic rate phs}
\lim\limits_{n\to+\infty}\left(d_\H(i,f^n(z))-\log (n)\right)=\log\left(\frac{|L(z)|}{\sqrt{I(z)}}\right),
\end{equation}
where $L$ is the function defined in Proposition \ref{prop:phs}.
\end{theorem}
\begin{proof}
Fix $z\in\H$. First, assume that $f$ is of extremal rate. By definition, this means that $|L(z)|=\lim_{n\to+\infty}(|f^n(z)|/n)$ exists in $(0,+\infty)$. Following a procedure similar to the previous proof, we have
\begin{eqnarray*}
&& d_\H(i,f^n(z))-\log (n)=\frac{1}{2}\log\frac{|f^n(z)+i|+|f^n(z)-i|}{|f^n(z)+i|-|f^n(z)-i|}-\log (n) \\
&=&\log\frac{|f^n(z)+i|}{n}+\log\left(1+\frac{|f^n(z)-i|}{|f^n(z)+i|}\right)-\frac{1}{2}\log(4\mathrm{Im}(f^n(z))).
\end{eqnarray*}
Taking limits as $n\to+\infty$, we obtain
\begin{equation}\label{eq:hyperbolic rate phs 1}
\lim\limits_{n\to+\infty}\left(d_\H(i,f^n(z))-\log (n)\right)=\log(|L(z)|)-\frac{1}{2}\log\left(\lim\limits_{n\to+\infty}\mathrm{Im}(f^n(z))\right).
\end{equation}
Recall that $f$ being of extremal rate is synonymous to $f$ being of finite shift. Therefore, $\lim_{n\to+\infty}\mathrm{Im}(f^n(z))$ exists in $(0,+\infty)$ and \eqref{eq:hyperbolic rate phs 1} is exactly \eqref{eq:hyperbolic rate phs}.

For the converse implication, assume that $f$ is not of extremal rate and thus is of infinite shift. Let $h$ be the Koenigs function of $f$ and without loss of generality, assume that $\Omega\vcentcolon=h(\H)\subset \H$. As we mentioned in Section \ref{sec:conformality}, $h$ is not conformal at infinity. Hence, in view of Lemma \ref{lm:conformality hyperbolic distance},
\begin{equation}\label{eq:hyperbolic rate phs 2}
\liminf\limits_{\H\ni w\to\infty}\left(d_\H(i,w)-d_\H(i,h(w))\right)=+\infty.
\end{equation}
As a result the actual limit is infinite and in place of $\H\ni w\to\infty$ we may take $f^n(z)$ since $f$ has infinity as its Denjoy--Wolff point. In addition, since $h$ is the Koenigs function of $f$, we have $h\circ f^n=h+n$. Consequently, \eqref{eq:hyperbolic rate phs 2} implies
\begin{equation}\label{eq:hyperbolic rate phs 3}
\lim\limits_{n\to+\infty}\left(d_\H(i,f^n(z))-d_\H(i,h(z)+n)\right)=+\infty.
\end{equation}
Using \eqref{eq:hypdistance in H}, we easily notice that
\begin{equation*}
\lim_{n \to +\infty}(d_{\H}(i,w+n)-\log(n)) = -\dfrac{1}{2}\log(\mathrm{Im}(w)), \qquad w \in \H.
\end{equation*}
In conclusion, applying this piece of information on \eqref{eq:hyperbolic rate phs 3}, we obtain
\begin{equation*}
\lim\limits_{n\to+\infty}\left(d_\H(i,f^n(z))-\log (n)\right)=+\infty,
\end{equation*}
which provides the desired equivalence.
\end{proof}

\begin{remark}
\label{rem:infinite shift rate}
Looking closely at the proof of the previous theorem, we may observe that in case $f$ is not of extremal rate (and thus its orbits have unbounded imaginary part), $\lim_{n\to+\infty}(|f^n(z)|/(n\sqrt{\mathrm{Im}(f^n(z))}))=+\infty$, which slightly strengthens Proposition \ref{prop:phs}. In addition, this fact shows that \eqref{eq:hyperbolic rate phs} is also satisfied even in the absence of extremal rate.
\end{remark}

\section{Composition operators}
\label{sec:comp-operators}
In recent literature, the rate of convergence to the Denjoy--Wolff point has been principally studied for continuous semigroups, and almost exclusively in the setting of the unit disc. For this reason, it makes sense to translate the extremal rates we defined and investigate how they are represented in the unit disc. As we will see in a moment, this translation automatically leads to a new description concerning the norms of composition operators.

Clearly, the settings of the unit disc and of the upper half-plane can be thought of as interchangeable by using M\"{o}bius transformations. Hence, the dynamical aspects that we have already described for $\H$ may be easily translated to $\D$. Let $f \vcentcolon \H \to \H$ be a non-elliptic self-map with Denjoy--Wolff point infinity, and choose $\tau \in \partial\D$. We define the conjugated map $g \vcentcolon \D \to \D$ as $g = S^{-1} \circ f \circ S$, where $S \vcentcolon \D \to \H$ is the M\"obius transformation given by
\begin{equation}
\label{eq:map-S}
S(z) = i\dfrac{\tau+z}{\tau-z}, \quad z \in \D, \quad S^{-1}(w) = \tau\dfrac{w-i}{w+i}, \quad w \in \H.
\end{equation}
We will say that $g$ is a \textit{conjugation} of $f$ in $\D$. For $z \in \D$ and $n \in \N$, we have that $g^n(z) = S^{-1}(f^n(w))$, where $w = S(z)$. In particular, since $S^{-1}(\infty) = \tau$, it is clear that $\lim_{n \to +\infty}g^n(z) = \tau$ for all $z \in \D$. For this reason, we say that $\tau$ is the Denjoy--Wolff point of $g$.

Due to \eqref{eq:map-S}, direct calculations show that 
\begin{equation}\label{eq:rate in D}
\lim_{n\to+\infty}(\lvert g^n(z)-\tau\rvert\lvert f^n(w)\rvert)=2, 
\end{equation}
for all $z\in\D$ and $w=S(z)$. This allows us to prove the following basic lemmas:

\begin{lemma}
\label{lm:hyperbolic D}
Let $f\vcentcolon\H\to\H$ be hyperbolic with Denjoy--Wolff point infinity and $g$ its conjugation in $\D$ with Denjoy--Wolff point $\tau\in\partial\D$. Let $\alpha=f'(\infty)>1$. Then, the following are equivalent:
\begin{enumerate}
\item[\textup{(a)}] $f$ is of extremal rate.
\item[\textup{(b)}] $\lim_{n\to+\infty}(\alpha^n\lvert g^n(z)-\tau\rvert)$ exists in $(0,+\infty)$ for some (and hence all) $z\in\D$.
\item [\textup{(c)}] $\lim_{n\to+\infty}[\alpha^n(1-\lvert g^n(z)\rvert)]$ exists in $(0,+\infty)$ for some (and hence all) $z\in\D$.
\end{enumerate}
\end{lemma}
\begin{proof}
Statements (a) and (b) are clearly equivalent due to Proposition \ref{prop:hyperbolic-rate-no-z}, Definition \ref{def:extremal hyp} and \eqref{eq:rate in D}. On the other hand, since $f$ is hyperbolic, for each $w\in\H$, the orbit $\{f^n(w)\}$ diverges to $\infty$ by an angle (i.e. the limit of $\arg(f^n(w))$ as $n \to +\infty$ exists in $(0,\pi)$). But M\"{o}bius transformations preserve angles and hence each orbit $\{g^n(z)\}$, $z\in\D$, converges to $\tau$ by an angle and in particular, non-tangentially. As a result, the limit $\lim_{n\to+\infty}[\lvert g^n(z)-\tau \rvert/(1-\lvert g^n(z)\rvert)]$ exists in $[1,+\infty)$, which in turn provides the equivalence between statements (b) and (c).
\end{proof}

\begin{lemma}
\label{lm:phs D}
Let $f\vcentcolon\H\to\H$ be parabolic of positive hyperbolic step with Denjoy--Wolff point infinity and $g$ its conjugation in $\D$ with Denjoy--Wolff point $\tau\in\partial\D$. Then, the following are equivalent:
\begin{enumerate}
\item[\textup{(a)}] $f$ is of extremal rate.
\item[\textup{(b)}] $\lim_{n\to+\infty}(n|g^n(z)-\tau|)$ exists in $(0,+\infty)$ for some (and hence all) $z\in\D$.
\item[\textup{(c)}] $\lim_{n\to+\infty}[n^2(1-|g^n(z)|)]$ exists in $(0,+\infty)$ for some (and hence all) $z\in\D$.
\end{enumerate}
\end{lemma}
\begin{proof}
Once again, the equivalence between (a) and (b) is straightforward. We will prove that (a) and (c) are equivalent as well. Fix $z\in\D$. Then, quick calculations show that $\mathrm{Im}(f^n(w))=(1-|g^n(z)|^2)/|g^n(z)-\tau|^2$, where $w=S(z)$. This leads to
\begin{equation}
\label{eq:finte shift in D}
n^2\left(1-|g^n(z)|\right)=\left(n|g^n(z)-\tau|\right)^2\frac{\mathrm{Im}(f^n(w))}{1+|g^n(z)|},
\end{equation}
for all $n\in\N$. Suppose, first, that $f$ is of extremal rate and thus of finite shift. Then $I(w)=\lim_{n\to+\infty}\mathrm{Im}(f^n(w))$ exists in $(0,+\infty)$. Hence, the equivalence between (a) and (b), together with \eqref{eq:finte shift in D}, imply that the limit $\lim_{n\to+\infty}(n^2(1-|g^n(z)|))$ exists in $(0,+\infty)$.

Finally, assume that $f$ is not of extremal rate. By Remark \ref{rem:infinite shift rate}, we have
\begin{equation*}
\lim_{n\to+\infty}\frac{|f^n(w)|}{n\sqrt{\mathrm{Im}(f^n(w))}}=+\infty.
\end{equation*}
Executing simple computations through the conformal transformation $S$, this translates to
\begin{equation*}
\lim_{n\to+\infty}\frac{|g^n(z)+\tau|}{n\sqrt{(1-|g^n(z)|^2)}}=+\infty.
\end{equation*}
Consequently, the limit $\lim_{n\to+\infty}(n^2(1-|g^n(z)|))$ necessarily equals $0$ and we are done.
\end{proof}

As evidenced by Lemmas \ref{lm:hyperbolic D} and \ref{lm:phs D}, our two notions of extremality are completely described by the asymptotic behavior of $1-|g^n(z)|$ when working in the unit disc. However, known results in the field of functional analysis relate such quantities with the norms of composition operators with respect to classical spaces of analytic functions. Due to this relation, we will now investigate the applications of our results in regard to composition operators.

Let $X$ be a Banach space of analytic functions in $\D$ and suppose that $g \vcentcolon \D \to \D$ is an analytic mapping. Then, the \textit{composition operator} $C_g$ induced by $g$ and acting on $X$ is given by $C_g(f) = f \circ g$, for $f \in X$.

For the purposes of this article, we will only deal with composition operators acting on the classical Hardy or Bergman spaces. We are going to need solely certain basic facts about these spaces. First of all, the \textit{Hardy space} $H^p\vcentcolon =H^p(\D)$, $p>0$, consists of all the holomorphic functions $f\vcentcolon\D\to\C$ such that
$$\sup\limits_{r\in(0,1)}\int_{0}^{2\pi}|f(re^{i\theta})|d\theta<+\infty.$$
We refer to \cite{Duren-Hp} for a complete exposition on these spaces. It is well-known that, due to Littlewood's Subordination Theorem \cite[Theorem 1.7]{Duren-Hp}, every composition operator is bounded from $H^p$ into itself, $p \geq 1$. In particular, the norm of the operator may be estimated through the following inequality:
\begin{lemma}
{\rm \cite[Corollary 3.7]{Cowen-Maccluer}}
\label{lm:hardy growth}
Let $g \vcentcolon \D \to \D$ be analytic. Then
\begin{equation}
\label{eq:hardy growth}
\left(\frac{1}{1-|g(0)|^2}\right)^{\frac{1}{p}}\le ||C_g||_{H^p} \le \left(\frac{1+|g(0)|}{1-|g(0)|}\right)^{\frac{1}{p}},
\end{equation}
where $||C_g||_{H^p}$ denotes the norm of the operator $C_g$ acting on the Hardy space $H^p$, $p \geq 1$.
\end{lemma}

Besides, the \textit{Bergman space} $A^p=A^p(\D)$, $p>0$, consists of all the holomorphic functions $f\vcentcolon\D\to\C$ satisfying
$$\int_{\D}|f(z)|^pdA(z)<+\infty,$$
where $dA(z)$ denotes the normalized Lebesgue area measure in $\D$. We mostly follow \cite{Duren-Apa} for this exposition. Once again, the boundedness of the composition operators acting on a Bergman space also follows from Littlewood's Subordination Theorem. Moreover, combining \cite[Theorem 11.6]{Zhu} and \cite[Theorem 1]{Vukotic} we may extract the following inequality about the norm of a composition operator acting on a Bergman space. 

\begin{lemma}
\label{lm:bergman growth}
Let $g \vcentcolon \D \to \D$ be analytic. Then
\begin{equation}
\label{eq:bergman growth}
\left(\frac{1}{1-|g(0)|^2}\right)^{\frac{2}{p}} \le ||C_g||_{A^p} \le \left(\frac{1+|g(0)|}{1-|g(0)|}\right)^{\frac{2}{p}},
\end{equation}
where $||C_g||_{A^p}$ denotes the norm of the operator $C_g$ acting on the Bergman space $A^p$, $p\ge1$.
\end{lemma}

We are now ready to continue with our two final results. The cornerstones for their proofs will be Lemmas \ref{lm:hyperbolic D} and \ref{lm:phs D}. Since the proofs are similar, we will omit the proof of the second one, which refines \cite[Corollaries 4.3(c) and 4.5(c)]{KTZ-Operators}, for the sake of avoiding repetition.

\begin{corollary}
\label{thm:norm hyp}
Let $f\vcentcolon\H\to\H$ be hyperbolic with Denjoy--Wolff point infinity. Let $\alpha=f'(\infty)>1$ and suppose that $g$ is a conjugation $f$ in $\D$. The following are equivalent:
\begin{enumerate}[\hspace{0.5cm}\rm(a)]
\item $f$ is of extremal rate;
\item There exists $C_1\vcentcolon =C_1(g)\ge1$ such that
\begin{equation*}
\dfrac{1}{C_1} \leq \dfrac{||C_{g^n}||_{H^p}^p}{\alpha^n} \leq C_1,\quad \textup{for all }n\in\N \textup{ and all }p\ge1;
\end{equation*}
\item There exists $C_2\vcentcolon =C_2(g)\ge1$ such that
\begin{equation*}
  \dfrac{1}{C_2} \leq \dfrac{||C_{g^n}||_{A^p}^p}{\alpha^{2n}} \leq C_2, \quad\textup{for all }n\in\N \textup{ and all }p\ge1.
\end{equation*}
\end{enumerate}
\end{corollary}
\begin{proof}
We will first prove that (a) and (b) are equivalent. Each $g^n$ is an analytic self-map of $\D$. Therefore, Lemma \ref{lm:hardy growth} is applicable and we see that 
\begin{equation}\label{eq:hardy1}
\frac{1}{2} \leq ||C_{g^n}||_{H^p}^p(1-|g^n(0)|) \leq 2, \quad\textup{for all }n\in\N \textup{ and all }p\ge1.
\end{equation}
Then, Lemma \ref{lm:hyperbolic D} provides at once the desired equivalence. The same process, but this time using Lemma \ref{lm:bergman growth}, provides the equivalence between (a) and (c).
\end{proof}

\begin{corollary}
\label{thm:norm phs}
Let $f\vcentcolon\H\to\H$ be parabolic of positive hyperbolic step with Denjoy--Wolff infinity. Suppose that $g$ is a conjugation of $f$ in $\D$. The following are equivalent:
\begin{enumerate}[\hspace{0.5cm}\rm(a)]
\item $f$ is of extremal rate
\item There exists $C_1\vcentcolon= C_1(g)\ge1$ such that
\begin{equation*}
\frac{1}{C_1}\le \frac{||C_{g^n}||_{H^p}^p}{n^2} \le C_1, \quad\textup{for all }n\in\N \textup{ and all }p\ge1; 
\end{equation*}
\item There exists $C_2\vcentcolon=C_2(g)\ge1$ such that
\begin{equation*}
\frac{1}{C_2} \le \frac{||C_{g^n}||_{A^p}^p}{n^4} \le C_2, \quad \textup{for all }n\in\N \textup{ and all }p\ge1.
\end{equation*}
\end{enumerate}
\end{corollary}

\end{document}